\documentclass[10pt]{article}
\usepackage[utf8]{inputenc}
\usepackage{mathtools,
        mathrsfs,amssymb,amsthm, amsmath,
        bm}
\usepackage{comment}
\usepackage{hyperref}
\usepackage{url}
\usepackage[a4paper, left = 1in, right = 1in, top =1in, bottom = 1in]{geometry} 
\usepackage{dsfont}   
\usepackage[british]{babel}

\newtheorem{thm}{Theorem}
\newtheorem{pro}{Proposition}

\makeatletter
\renewenvironment{proof}[1][\proofname] {\par\pushQED{\qed}\normalfont\topsep6\p@\@plus6\p@\relax\trivlist\item[\hskip\labelsep\bfseries#1\@addpunct{.}]\ignorespaces}{\popQED\endtrivlist\@endpefalse}
\makeatother

\def\[#1\]{\begin{align*}#1\end{align*}}

\newcommand{\R}{\mathbb{R}}
\newcommand{\Z}{\mathbb{Z}}
\newcommand{\N}{\mathbb{N}}
\newcommand{\C}{\mathbb{C}}
\newcommand{\ceq}{\coloneqq}

\newcommand{\eps}{\varepsilon}
\newcommand{\E}{\mathbb{E}}
\renewcommand{\P}{\mathbb{P}}
\def\B{\mathscr{B}}

\newcommand{\ltx}[1]{\ \text{#1}}

\newcommand{\br}{\overline}
\newcommand{\I}{\mathds{1}}
\newcommand{\df}{\mathop{}\!\mathrm{d}}
\newcommand{\M}
{\mathbb{M}}
\newcommand{\F}
{\mathscr{F}}

\newcommand{\iu}{\mathrm{i}}
\newcommand{\A}{\mathscr{A}}
\newcommand{\sd}{\,\Delta\,}

\begin{document}

\title{Change of Measures for Spectral Stochastic Integrals}
\author{Yu-Lin Chou\thanks{Yu-Lin Chou, Institute of Statistics, National Tsing Hua University, Hsinchu 30013, Taiwan,  R.O.C.; Email: \protect\url{y.l.chou@gapp.nthu.edu.tw}.}}
\date{}
\maketitle

\begin{abstract}
Under mild conditions, it is possible to obtain, from almost purely
measure-theoretic considerations and without any specific reference
to stochastic processes, a change-of-measures result, resembling the
usual Radon-Nikod{\'y}m change of measures, associated with a variant
of stochastic integration for a spectral representation of covariance
stationary processes; the ideas are naturally embedded in the Hilbert
space theory of $L^{2}$ spaces. The intended main contribution, including
a complete proof of change of measures for spectral stochastic integrals,
is the refined, self-contained developments of spectral stochastic
integration toward change of measures.\\

{\noindent \textbf{Keywords:}} change of measures; orthogonal stochastic measures; spectral representation for covariance stationary processes; stochastic integration

{\noindent \textbf{MSC 2020:}} 60A10; 60H05; 60G10; 37M10
\end{abstract}

\section{Introduction }

For and only for our motivating account, a stochastic process is called
a \textit{covariance stationary process} if and only if the process
is a random element of $\C^{\Z}$ with components i) having finite
second moment, i.e. being $L^{2}$, ii) having a constant mean, and
iii) having the property that the covariance of any pair of the components
depends at most on the difference of their indexes. A well-known (classical)
version of Herglotz representation theorem (e.g. Section 1, Chapter
6, Shiryaev \cite{s}) asserts that, for every centered covariance
stationary process, there is some complex measure on the Borel sigma-algebra
$\B_{[-\pi, \pi[}$ of $\R$ relativized to $[-\pi, \pi[$, whose
total variation measure is concentrated on $[-\pi, \pi[$, such that
for every $n \in \Z$ the covariance function of the process at $n$
equals the integral of the function $\lambda \mapsto e^{\iu \lambda n}$
over $[-\pi, \pi[$ with respect to the complex measure. 

The Herglotz spectral representation result suggests, in a mathematically
natural way, if, for every $n \in \Z$, one can express the $n$th
component, rather than the covariance function, of a centered covariance
stationary process in terms of (modulo some underlying probability
measure) an integral of the function $\lambda \mapsto e^{\iu \lambda n}$
over $[-\pi, \pi[$ in a suitable sense. To this end, special attention
would be required as i) such an integral has to be ``stochastic”
and as ii) it is not clear that a naive pointwise definition could
always circumvent the possibilities of encountering functions that
are not of bounded variation. These well-known potential difficulties,
together with the desire to obtain a spectral representation result
for the components of a covariance stationary process parallel to
the aforementioned Herglotz represntation for the covariance function
of the process, thus logically motivate the developments of a type
of stochastic integration, which we refer to as \textit{spectral stochastic integration}.
The modifier ``spectral” signifies the purpose-specific aspect
of that kind of stochastic integration concerning us\footnote{Exploring the connections between spectral stochastic integration and other existing notions of stochastic integration is outside the scope of the present paper. Spectral stochastic integration might be likened to Paley-Wiener-Zygmund integration with a careful distinction, the latter of which is also concerned with assigning a suitable sense to integrating a ``deterministic" function with respect to a stochastic object.}. 

To the best of the author's knowledge, the corresponding theory to
spectral stochastic integration is \textit{de facto} scattered in
the related literature with relatively incomplete or cursory characterizations.
One of the most complete treatments of spectral stochastic integration
without loss of mathematical rigor known to the author would be Shiryaev
\cite{s}, which leverages the fact that every $L^{2}$ space of $\C$-valued
functions is a Hilbert space without bringing in too many context-unnecessary
concepts and results, and furnishes an outline of the framework. The
other would be Gikhman and Skorokhod \cite{gs}. Although a change-of-measures
result for spectral stochastic integration is given therein (with
only a brief proof), their definition for spectral stochastic integration
depends on several deeper results in analysis, and hence might unintentionally
obscure the simple nature of spectral stochastic integration. Moreover,
in some directions (with other directions fixed) their definition
is narrower than Shiryaev \cite{s}. 

Based on what is outlined in Shiryaev \cite{s}, the present paper
intends to complete a corner of the theory of spectral stochastic
integration by redeveloping a systematic, unified, and non-redundant
treatment to proving a natural change-of-measures result for spectral
stochastic integrals that resembles the usual Radon-Nikod{\'y}m version.
Our proof for change of measures, intended as a complete one, is different
than the proof idea sketched in the aforementioned Gikhman and Skorokhod
\cite{gs}. Although a working knowledge in stochastic processes would
be helpful in appreciating the theory of spectral stochastic integration,
we motivate the building concepts of such integration and arrange
the developments so that literally no working knowledge in stochastic
processes is demanded. As a whole, these novel treatments of ``known”
ideas are also intended both as a compact, citable reference and as
a contribution inclined to the pedagogical side.

\section{Result }

\subsection{Preliminary Developments}

To assign a suitable sense to that kind of stochastic integration
serving the purpose of ``spectrally” representing covariance stationary
processes, i.e. to our spectral stochastic integration, it surprisingly
suffices to employ a few natural requirements. Let $\Omega$ be a
probability space with probability measure $\P$; let $S \subset \Omega$
be nonempty; let $\A$ be an algebra of subsets of $S$. By an \textit{orthogonal elementary stochastic measure},
which is the building block for our spectral stochastic integral,
is meant a family $\M \equiv (\M(A))_{A \in \A}$ of complex random
variables $\in L^{2}(\P)$ on $\Omega$ such that i) $\M(\varnothing) = 0$
a.s.-$\P$, ii) $\M(A_{1} \cup A_{2}) = \M(A_{1}) + \M(A_{2})$ a.s.-$\P$
and $\E \M(A_{1})\br{\M(A_{2})} = 0$ for all disjoint $A_{1}, A_{2} \in \A$,
and iii) $A_{1}, A_{2}, \dots \in \A$ being disjoint and $\cup_{n \in \N}A_{n} \in \A$
imply $\E |\M(\cup_{n}A_{n}) - \sum_{j=1}^{n}\M(A_{j})|^{2} \to 0$.
Since $(X,Y) \mapsto \int_{\Omega} X\br{Y} \df \P = \E X\br{Y}$ is
an inner product for $L^{2}(\P)$, the defining properties of an orthogonal
elementary stochastic measure explain the terminology. Since every
component of $\M$ is by definition $L^{2}$, we can define the function
$m: A \mapsto \E|\M(A)|^{2}$ on $\A$. To justify the existence of
an orthogonal elementary stochastic measure, we give a stronger result
with a handy construction without any reference to the theory of stochastic
processes:

\begin{pro}\label{pro1}
For every probability space $(\Omega, \F, \P)$ there is some orthogonal elementary stochastic measure on $\Omega$.
\end{pro}

\begin{proof}

Since $\F$ is also an algebra, we consider the family $(\I_{A})_{A \in \F}$
of the indicators of sets of $\F$. If $\M(A) \equiv \M(A, \cdot) \ceq \I_{A}$
on $\Omega$ for all $A \in \F$, it is readily checked from the usual
properties of indicator functions that $\M$ is an orthogonal elementary
stochastic measure on $\Omega$.

Indeed, as a side remark, the probability measure $\P$ plays the
role of $m$ above. \end{proof}

Now the definition of an orthogonal elementary stochastic measure
ensures the existence of the Carath{\'e}odory extension $M$ for
$m$ to $\sigma(\A)$. This observation enables us to employ the completeness
of $L^{2}(S, \sigma(\A), M)$ to define our spectral stochastic integrals.
The function $M$ will be referred to as the \textit{structural function}\footnote{This terminology would be in a sense self-evident once we recall that the spectral measure present in the Herglotz representation for the covariance function of a covariance stationary process happens to play the role of the structural function for  an orthogonal stochastic measure employed to represent the components of the covariance stationary process.}
for $\M$; and, later on, a triple of the form $(\M, M, \A)$ declares
$\M$ to be an orthogonal elementary stochastic measure with $M$
being its structural function defined on the sigma-algebra generated
by a given algebra $\A$ of sets.

Given any $\A$-simple function $f: S \to \C$ of the form
$f = \sum_{j=1}^{n}a_{j}\I_{A_{j}}$ where $a_{1}, \dots, a_{n} \in \C$
are distinct and where $A_{1}, \dots, A_{n} \in \A$, we define \[
\int_{S} f \df \M \ceq \sum_{j=1}^{n}a_{j}\M(A_{j}),
\] and refer to $\int_{S}f\df \M$ as the spectral stochastic integral
of $f$ with respect to $\M$. The convention of omitting the domain
of integration applies here, and we remark that $\int f \df \M \in L^{2}(\P)$
for all $\A$-simple $f: S \to \C$ by the very definition of $\M$.
The finite additivity and the orthogonality of $\M$ imply that $\E|\M (A \cap A')|^{2} = \E \M(A)\br{\M(A')}$
for all $A,A' \in \A$, which follows from a consideration over the
partition $A = (A \cap A') \cup (A \setminus A')$ for $A$ and the
same partition for $A'$; with more notation it then holds that \[
\E \bigg( \int f \df \M \bigg)\bigg(\br{ \int g \df \M }\bigg) = \int f\br{g} \df M
\]for all $\A$-simple $f,g: S \to \C$. Thus the space of all $\A$-simple
functions $f: S \to \C$ in $L^{2}(M)$ is inner-product homomorphic
to the space of their spectral stochastic integrals with respect to
$\M$. 

To extend the definition of our spectral stochastic integration for
arbitrary elements of $L^{2}(M)$, we claim that $\A$-simple functions
$S \to \C$ are $L^{2}$-dense in $L^{2}(M)$. This is not immediate
as the ``measurability” of the approximating sequence of simple
functions is now restricted to the algebra $\A$; the restriction
is reasonable as we have thus far defined our spectral stochastic
integration for and only for $\A$-simple functions. Fortunately,
the do-ability is not so covert; we begin by showing that for every
$B \in \sigma(\A)$ and every $\eps > 0$ there is some $A \in \A$
such that $M(A \sd B) < \eps$, which is a proposition usually left
as an exercise in textbooks. To see this, let $\mathscr{G}$ be the
collection of all such $B \in \sigma(\A)$. It is immediate that $\A \subset \mathscr{G}$
and that $\varnothing \in \mathscr{G}$. If $B \in \mathscr{G}$,
then the identity $A \sd B = A^{c} \sd B^{c}$ implies that $B^{c} \in \mathscr{G}$.
If $B_{1}, B_{2} \in \mathscr{G}$, then, since $(A_{1} \cup A_{2}) \sd (B_{1} \cup B_{2}) \subset (A_{1} \sd B_{1}) \cup (A_{2} \sd B_{2})$
for all $A_{1}, A_{2} \subset S$, choosing respectively the suitable
approximating $A_{1}, A_{2} \in \A$ for $B_{1}, B_{2}$ ensures that
$B_{1} \cup B_{2} \in \mathscr{G}$. If $B_{1}, B_{2}, \dots \in \mathscr{G}$,
then, as $M$ is a finite measure by the definition of $m$, there
is some $N \in \N$ such that $M(\cup_{n \geq N+1}B_{n})$ is as small
as desired; since $\cup_{n=1}^{N}B_{n} \in \mathscr{G}$, it follows
from the inclusion $\cup_{n \in \N}B_{n}\setminus \cup_{n=1}^{N}A_{n} \subset (\cup_{n=1}^{N}B_{n} \setminus \cup_{n=1}^{N}A_{n}) \cup (\cup_{n \geq N+1}B_{n})$
that $\cup_{n \in \N}B_{n} \in \mathscr{G}$. But then $\mathscr{G}$
is a sigma-algebra, and hence $\sigma(\A) \subset \mathscr{G}$. From
the identity $\int (\I_{A} - \I_{B})^{2} \df M = M(A \sd B)$, it
holds that the space of all $\A$-simple functions $S \to \C$ is
$L^{2}$-dense in that of all simple measurable functions $S \to \C$.
On the other hand, if $f \in L^{2}(M)$, and if $| \cdot |_{L^{2}(M)}$
denotes the $L^{2}$-norm of the space $L^{2}(M)$, then the $L^{2}$-denseness
of simple measurable functions in $L^{2}(M)$ and the (formal) triangle
inequality \[
|f - \psi|_{L^{2}(M)} \leq |f - \varphi|_{L^{2}(M)} + |\varphi - \psi|_{L^{2}(M)}
\](for $\psi$ being $\A$-simple and for $\varphi$ being simple measurable)
together imply the claim.

If we return to complete the definition of our spectral stochastic
integration, consider an arbitrary $f \in L^{2}(M)$. Since there
are some $\A$-simple functions $f_{1}, f_{2}, \dots: S \to \C$ such
that $|f_{n} - f|_{L^{2}(M)} \to 0$, the sequence $(f_{n})_{n \in \N}$
is $L^{2}$-Cauchy. But, if $| \cdot |_{L^{2}(\P)}$ denote the $L^{2}$-norm
of the space $L^{2}(\P)$, then the fact that \[
\bigg| \int f_{n} \df \M - \int f_{m} \df \M \bigg|_{L^{2}(\P)} = \bigg| f_{n} - f_{m} \bigg|_{L^{2}(M)}
\]for all $n,m \in \N$ and the completeness of $L^{2}(\P)$ jointly
imply that the sequence $(\int f_{n} \df \M)_{n \in \N}$ converges
in the space $L^{2}(\P)$ in the corresponding $L^{2}$ sense. The
spectral stochastic integral of $f$, denoted $\int f \df \M$, is
then defined as the $L^{2}$-limit of the sequence $(\int f_{n} \df \M)$.
Indeed, intuitively, the spectral stochastic integral of $f$ is invariant
in the choice of $(f_{n})$ as the principal ingredient of the above
construction is still a simple measurable $L^{2}$-approximating sequence
$(\varphi_{n})$ for $f$. We have completed the definition, based
on Shiryaev \cite{s}, of spectral stochastic integration with respect
to an orthogonal elementary stochastic measure for elements of the
space of $\C$-valued functions that are square-integrable with respect
to the structural function for the orthogonal elementary stochastic
measure.

\subsection{Change of Measures}

In this particular paragraph, we use the same notation as in the previous
subsection. If $g \in L^{2}(M)$, then $\I_{B}g \in L^{2}(M)$ for
all $B \in \sigma(\A)$, and our definition of a spectral stochastic
integral implies that $\int \I_{B}g \df \M \in L^{2}(\P)$ for all
$B \in \sigma(\A)$.

Now we prove the change-of-measures result for our spectral stochastic
integration:

\begin{thm}\label{thm}

Let there be given a probability space with probability measure $\P$;
let $\A$ be an algebra of subsets of a given subset of the probability
space; let $(\M_{1}, M_{1}, \A)$ be an orthogonal elementary stochastic
measure on the probability space; let $g \in L^{2}(M_{1})$. If $\M_{2}(B) \ceq \int \I_{B}g \df \M_{1}$
for all $B \in \A$, and if $m_{2}: B \mapsto \E |\M_{2}(B)|^{2}$
on $\A$, then i) the family $\M_{2} \equiv (\M_{2}(B))_{B \in \sigma(\A)}$
is an orthogonal elementary stochastic measure on the given probability
space with the Carath{\'e}odory extension $M_{2}$ of $m_{2}$ to
$\sigma(\A)$ being its structural function; and ii) \[
\int f \df \M_{2} = \int f g \df \M_{1} \,\,\ltx{a.s.-}\P
\]for all $f \in L^{2}(M_{2})$, where the $\P$-null set of the points
at which the equality possibly fails may depend on the choice of $f$. 

\end{thm}

\begin{proof}

Indeed, as readily seen from our definition of spectral stochastic
integration, a spectral stochastic integration operator acting on
the vector space $L^{2}$ with respect to the structural function
enjoys linearity and preserves inner product.

To prove i), we first observe that $\M_{2}(\varnothing) = 0$. The
linearity of a spectral stochastic integration operator implies $\M_{2}(A_{1} \cup A_{2}) = \int\I_{A_{1}}g \df \M_{1} + \int \I_{A_{2}}g \df \M_{1} = \M_{2}(A_{1}) + \M_{2}(A_{2})$
for all disjoint $A_{1}, A_{2} \in \A$. The orthogonality of $\M_{2}$
follows from the trivial equality $\I_{A_{1} \cap A_{2}}|g|^{2} = 0$
for all disjoint $A_{1}, A_{2} \in \A$ and from the inner-product-preserving
property of a spectral stochastic integration operator. To see the
countable additivity in the $L^{2}(\P)$ sense only takes the following
(formal) relations \[
\M_{2}(\cup_{n}A_{n}) - \sum_{j=1}^{n}\M_{2}(A_{j})
&= \int \I_{\cup_{n}A_{n}}g \df \M_{1} - \sum_{j=1}^{n}\int \I_{A_{j}}g \df \M_{1}\\
&\leq \int \I_{\cup_{n}A_{n} \setminus \cup_{j=1}^{n}A_{j}} g \df \M_{1};\\
\bigg| \int \I_{\cup_{n}A_{n} \setminus \cup_{j=1}^{n}A_{j}} g \df \M_{1} \bigg|^{2}_{L^{2}(\P)} 
&= \int \I_{\cup_{n}A_{n} \setminus \cup_{j=1}^{n}A_{j}} |g|^{2} \df M_{1}\\
&\to 0,
\]where we have acknowledged the linearity and the norm-preserving properties
of a spectral stochastic integration operator and the monotone convergence
theorem. By the argument present in the beginning of the present subsection,
it follows that $\M_{2}$ is an orthogonal elementary stochastic measure
on the given probability space. Now the definition of $m_{2}$ and
a Carath{\'e}odory extension applied to $m_{2}$ together complete
the proof of i).

To finish the whole proof, we bring in the fact that our spectral
stochastic integration operators enjoy, from the linearity and the
norm-preserving property, a ``continuity” property in the sense
that a sequence of suitable integrands converging in the suitable
$L^{2}$ sense to an $L^{2}$ integrand implies the convergence of
the corresponding sequence of spectral stochastic integrals in the
suitable $L^{2}$ sense to the spectral stochastic integral of the
limiting integrand. Let $f \in L^{2}(M_{2})$. Then, as was shown,
there are some $\A$-simple $f_{1}, f_{2}, \dots$ such that $f_{n} \to f$
in $L^{2}$; so $\int f_{n} \df \M_{2} \to \int f \df \M_{2}$ in
$L^{2}(\P)$. Since \[
M_{2}(B) 
&= |\M_{2}(B)|^{2}_{L^{2}(\P)}\\ 
&= \bigg| \int \I_{B}g \df \M_{1} \bigg|^{2}_{L^{2}(\P)}\\
&= \int \I_{B}|g|^{2} \df M_{1}
\]for all $B \in \sigma(\A)$, we have \[
\bigg| \int f_{n} \df \M_{2}  - \int f \df \M_{2} \bigg|_{L^{2}(\P)}^{2}
&= \bigg| \int (f_{n} - f) \df \M_{2} \bigg|_{L^{2}(\P)}^{2}\\
&= \bigg| f_{n} - f \bigg|_{L^{2}(M_{2})}^{2}\\
&= \int |f_{n} - f|^{2} \df M_{2}\\
&= \int |f_{n} - f|^{2}|g|^{2}\df M_{1};
\]so $(f_{n}-f)g \to 0$ in $L^{2}(M_{1})$, and the continuity property
of our spectral stochastic integration implies that \[
\bigg| \int (f_{n} - f)g \df \M_{1} \bigg|_{L^{2}(\P)} \to 0.
\] But \[
\int f_{n} \df \M_{2} = \int f_{n} g \df \M_{1}
\]for all $n \in \N$ by the definition of $\M_{2}$ and the linearity
of our spectral stochastic integration operators; the essential uniqueness
of $L^{2}$-limit then completes the proof.\end{proof}

\end{document}